\documentclass[11pt]{amsart}

\usepackage[a4paper,hmargin=3.5cm,vmargin=4cm]{geometry}
\usepackage{amsfonts,amssymb,amscd,amstext}
\usepackage{graphicx}
\usepackage[dvips]{epsfig} 
\usepackage{quoting}
\usepackage{hyperref} 

\usepackage{fancyhdr}
\input xy
\xyoption{all}

\usepackage{enumerate}
\usepackage{titlesec}
\usepackage{mathrsfs}

\titleformat{\section}
{\filcenter\bfseries} {\thesection{.}}{0.2cm}{}
\titleformat{\subsection}[runin]
{\bfseries} {\thesubsection{.}}{0.15cm}{}[.]
\titleformat{\subsubsection}[runin]
{\em}{\thesubsubsection{.}}{0.15cm}{}[.]

\usepackage[up,bf]{caption}

\newtheorem{theorem}{Theorem}[section]

\newtheorem{lemma}[theorem]{Lemma}

\newtheorem{corollary}[theorem]{Corollary}

\theoremstyle{definition}
\newtheorem{definition}[theorem]{Definition}

\newtheorem{conjecture}[theorem]{Conjecture}

\newtheorem{problem}[theorem]{Problem}

\numberwithin{equation}{section}
\numberwithin{figure}{section}

\newcommand\Cscr{\mathscr{C}}

\newcommand\B{\mathbb{B}}
\newcommand\C{\mathbb{C}}
\newcommand\D{\overline{\mathbb D}}

\renewcommand\D{\mathbb D}

\newcommand\N{\mathbb{N}}
\renewcommand\P{\mathbb{P}}

\newcommand\R{\mathbb{R}}

\renewcommand\b{\mathbb{B}}
\renewcommand\c{\mathbb{C}}

\renewcommand\d{\mathbb D}

\renewcommand\r{\mathbb{R}}

\newcommand\dist{\mathrm{dist}}

\def\dist{\mathrm{dist}}

\usepackage{color}

\begin{document}


\fancyhead[LO]{Complete minimal surfaces with Cantor ends}
\fancyhead[RE]{A.\ Alarc\'on}
\fancyhead[RO,LE]{\thepage}

\thispagestyle{empty}


\begin{center}
{\bf\Large Complete minimal surfaces with Cantor ends\\in minimally convex domains
}

\medskip

%
%
{\bf Antonio Alarc\'on}
\end{center}

\medskip

\begin{center}
{\em Dedicated to Nikolai Nadirashvili on the occasion of his seventieth birthday}
\end{center}

%
%
\medskip

\begin{quoting}[leftmargin={7mm}]
{\small
\noindent {\bf Abstract}\hspace*{0.1cm}
We survey the recent history of the conformal Calabi-Yau problem
consisting in determining the complex structures admitted by complete bounded minimal surfaces in $\R^3$. Moreover, we prove that for any minimally convex domain $\Omega$ in $\R^3$ and any compact Riemann surface $R$ there is a Cantor set $C$ in $R$ whose complement $R\setminus C$ is the complex structure of a complete proper minimal surface in $\Omega$.

\noindent{\bf Keywords}\hspace*{0.1cm} 
Minimal surface, Riemann surface, complete surface, minimally convex domain, Calabi-Yau problem


\noindent{\bf Mathematics Subject Classification (2020)}\hspace*{0.1cm} 
53A10, 53C42, 30Fxx
}
\end{quoting}


\section{The conformal Calabi-Yau problem: a brief survey}\label{sec:survey}

\noindent
In 1965, motivated by the well known fact that  there are no compact minimal submanifolds of the Euclidean space $\R^n$ $(n\ge 3)$, Calabi made the following conjectures (see \cite[p.\ 170]{Calabi1965Conjecture}):
%
%
\begin{conjecture}\label{co:Calabi}
\begin{enumerate}[\rm (a)]
\item There is no complete bounded minimal hypersurface in $\R^n$.

\smallskip

\item There is no complete nonflat minimal hypersurface in $\R^n$ with a bounded projection to an $(n-2)$-dimensional affine subspace.
\end{enumerate}
\end{conjecture}

An immersed submanifold $u:M\to\R^n$ is {\em complete} if the Riemannian metric $u^*ds^2$ induced on $M$ by pulling back the Euclidean metric $ds^2$ on $\R^n$ via the immersion $u$ is a complete metric on $M$: geodesics go on indefinitely. As a consequence of the Hopf-Rinow theorem, the immersion $u$ is complete if and only if the Euclidean length of $u\circ\gamma:[0,1)\to\R^n$ is infinite for every divergent path $\gamma:[0,1)\to M$. Completeness is a natural and very standard condition to impose on a Riemannian manifold when one is interested in its global properties.

Part (b) in Conjecture \ref{co:Calabi} is clearly more ambitious than part (a), which was promoted also by Chern in \cite[p.\ 212]{Chern1966BAMS}. Nothing seems available in the literature about Conjecture \ref{co:Calabi} for $n\ge 4$ (see Assimos, B\'ek\'esi, and Gentile \cite{AssimosBekesiGentile} for some remarks in higher codimension). The  case $n=3$ has been an active focus of research since the 1980s, and it still continues to provide new results and insights. In this section we shall briefly discuss what we know about the existence and properties of complete bounded minimal surfaces in $\R^3$. 
Surfaces in this paper are assumed to be connected. Recall that a $\Cscr^2$ immersed surface $u:M\to\R^3$ is {\em minimal} if it is a stationary point of the area functional for all compactly supported variations; i.e.,  for any smoothly bounded compact domain $D\subset M$ and any smooth variation $u^t$ of $u$ fixing the boundary $bD$ of $D$, the first variation of area at $t=0$ vanishes: $\frac{d}{dt}\big|_{t=0}{\rm Area}(u^t(D))=0$. By a theorem from 1776 due to Meusnier, $u$ is minimal if and only if its mean curvature 
vanishes identically. Further, this happens if and only if $u$ is harmonic in the Riemannian metric $u^*ds^2$. So, if $M$ is an open Riemann surface then a conformal (i.e., angle preserving) immersion $M\to\r^3$ is minimal if and only if it is a harmonic map (see, e.g., \cite[Theorem 2.3.1]{AlarconForstnericLopez2021Book}). For background on minimal surfaces, we refer, e.g., to the classical monographs by Osserman \cite{Osserman1986} and Lawson \cite{Lawson1980}, and the more recent ones by Meeks and P\'erez \cite{MeeksPerez2012Survey} and Alarc\'on, Forstneri\v c, and L\'opez \cite{AlarconForstnericLopez2021Book}. Our focus in this brief survey is mainly on the conformal properties of complete bounded minimal surfaces in $\R^3$. In particular, we shall put the emphasis on which open Riemann surfaces are known to be the underlying complex (= conformal) structure of such a surface. Without the aim of offering a comprehensive survey, we refer the most interested readers to \cite[Chapter 7]{AlarconForstnericLopez2021Book} for a more detailed exposition. 

Conjecture \ref{co:Calabi} (b) was settled in the negative by Jorge and Xavier, who proved in 1980 that there is a complete nonflat conformal minimal immersion $\d=\{z\in\c\colon |z|<1\}\to\R^3$ having a bounded component function, and hence with the range contained between two parallel planes (i.e., with bounded projection into an affine line) \cite{JorgeXavier1980AM}. The groundbreaking counterexample by Jorge and Xavier has been generalized in several ways. We point out that Alarc\'on, Fern\'andez, and L\'opez established in \cite{AlarconFernandezLopez2012CMH} that every nonconstant harmonic function on an open Riemann surface $M$ is a component function of a complete nonflat conformal minimal immersion $M\to\R^3$. The following characterization result follows \cite{AlarconFernandezLopez2012CMH}.

\begin{theorem}\label{th:cCY-slab}
An open Riemann surface $M$ is the underlying complex structure of a complete nonflat minimal surface in $\R^3$ contained between two parallel planes if and only if $M$ admits a nonconstant bounded harmonic function.
\end{theorem}

In 1982, by which time the Jorge-Xavier paper had appeared, Yau \cite[Problem 91]{Yau1982} revisited Conjecture \ref{co:Calabi} (a) and turned it into a question which became known in the literature as the {\em Calabi-Yau problem}. This was settled in 1996 by the following landmark result due to Nadirashvili \cite{Nadirashvili1996IM}.
%
%
\begin{theorem}\label{th:Nadirashvili}
There is a complete conformal minimal immersion $\D\to\R^3$ with the range contained in the unit ball $\B=\{x\in\R^3\colon |x|<1\}\subset\R^3$. Furthermore, there is such an immersion with negative Gaussian curvature.
\end{theorem}

The second part of Nadirashvili's theorem settled the problem posed by Hadamard \cite{Hadamard1898} in 1898 asking whether there are complete bounded immersed surfaces in $\R^3$ with negative curvature. This question remains open and is likely very difficult for embedded surfaces (see, e.g., Calabi \cite[p.\ 170]{Calabi1965Conjecture}, Rozendorn \cite[\textsection2.3]{Rozendorn1992}, or Ghomi \cite[Problem 7.2]{Ghomi2017}). It is on the other hand known that there are complete bounded embedded surfaces in $\R^3$ with non-positive Gaussian curvature (see Rozendorn \cite{Rozendorn1961} and Xavier \cite{Xavier1984}). Nadirashvili's surfaces in Theorem \ref{th:Nadirashvili} cannot be embedded. In fact, a complete embedded minimal surface in $\R^3$ of finite genus and at most countably many ends is proper in space, hence unbounded, as was shown by Meeks, P\'erez, and Ros \cite{MeeksPerezRos2021Duke}. The same result for surfaces of finite topology was previously established by Colding and Minicozzi \cite{ColdingMinicozzi2008AM}. It still remains open the question whether there is a complete bounded embedded minimal surface (of infinite genus or with uncountably many ends) in $\R^3$.

Nadirashvili's example in \cite{Nadirashvili1996IM} is somewhat similar to that of Jorge and Xavier in \cite{JorgeXavier1980AM}: both constructions use the Weierstrass representation formula of minimal surfaces (see, e.g., \cite[\textsection2.3]{AlarconForstnericLopez2021Book}), the L\'opez-Ros deformation (a simple method introduced in \cite{LopezRos1991JDG} to deform a given conformal minimal immersion $\d\to\r^3$ preserving a component function), and the classical Runge approximation theorem for holomorphic functions (see, e.g., \cite[Chapter VIII]{Conway1973}) on a labyrinth of compact sets in $\D$. Nadirashvili's construction follows from a recursive procedure, and the fundamental new idea is the following clever application of Pythagoras' theorem. Assume that we have a compact minimal disc $S$ in $\R^3$ that is contained in a ball $r\B=\{x\in\R^3\colon |x|<r\}$ centered at the origin and of a certain radius $r>0$. If we apply a deformation to $S$ that is small outside a neighborhood of the boundary  $bS$ and pushes each boundary point $x\in bS$ a distance approximately $s>0$ in a direction approximately orthogonal to the position vector of $x$ itself in $\R^3$, then it is clear that we produce an increase of an amount of approximately $s$ of the boundary distance from a fixed interior point of the surface. However, by Pythagoras' theorem, the resulting surface is contained in the ball of radius $\sqrt{r^2+s^2}$, so the increase of its extrinsic diameter is of the order of $s^2$. With this idea in mind and using the aforementioned tools in a rather involved way, Nadirashvili was able to construct a sequence of conformal minimal immersions $u_j:\overline\D=\{z\in\C\colon |z|\le1\}\to\R^3$ and numbers $r_j>0$, $\rho_j>0$, and $\epsilon_j>0$ satisfying the following conditions:
\begin{enumerate}[A]
\item[$\bullet$] $u_j(0)=0$ for all $j\in\N$.

\smallskip
\item[$\bullet$] $|u_j(z)-u^{j-1}(z)|<\epsilon_j$ for all $z\in (1-\epsilon_j)\overline\D$ for all $j\ge 2$.

\smallskip
\item[$\bullet$] $u_j(\overline\D)\subset r_j\B$  for all $j\in\N$.

\smallskip
\item[$\bullet$] $\dist_{u_j}(0,b\overline\D)>\rho_j$  for all $j\in\N$, where $\dist_{u_j}$ denotes the distance function associated to the metric $u_j^*ds^2$ induced on $\overline\D$ by pulling back the Euclidean metric $ds^2$ on $\R^3$ via the immersion $u_j$, and $b\overline\D=\{z\in\c\colon |z|=1\}$.

\smallskip
\item[$\bullet$] $r_j=\sqrt{r_{j-1}^2+1/j^2}$, $\rho_j=\rho_{j-1}+1/j$, and $\epsilon_j<\epsilon_{j-1}/2$ for all $j\ge 2$.
\end{enumerate}
The key observation is that 
\[
	\lim_{j\to\infty}r_j=r<+\infty \quad\text{and}\quad \lim_{j\to\infty}\rho_j=+\infty.
\] 
It is then easily seen that if the sequence $\epsilon_j$ decreases to $0$ sufficiently fast, then there is a limit map $u=\lim_{j\to\infty}u_j:\D\to\R^3$ that is a complete conformal minimal immersion with $u(\D)\subset r\B$. A brief explanation of Nadirashvili's construction can be found in \cite[\textsection7.1]{AlarconForstnericLopez2021Book}, while complete details are available in Nadirashvili's original paper \cite{Nadirashvili1996IM} as well as in the note by Collin and Rosenberg \cite{CollinRosenberg1999BSC}.

It was again Yau in his 2000 millenium lecture who, in view of Nadirashvili's theorem, proposed several questions about the topological, geometric, and conformal properties of complete bounded minimal surfaces in $\R^3$; see \cite[p.\ 360]{Yau2000AMS} and \cite[p.\ 241]{Yau2000AJM}. This contributed to motivate the study of these surfaces, giving rise to a large literature; we refer once again to \cite[Chapter 7]{AlarconForstnericLopez2021Book} for a survey and references. Indeed, Nadirashvili's method has been the seed of several construction techniques, leading to a wide variety of examples. In particular, this method, together with the technique developed by Morales in \cite{Morales2003GAFA} for constructing a proper conformal minimal disc in $\R^3$, has been generalised to the construction of complete properly immersed minimal surfaces with arbitrary topology in any given convex domain of $\R^3$; see the paper by Ferrer, Mart\'in, and Meeks \cite{FerrerMartinMeeks2012AM} and the references therein. However, despite the remarkable flexibility of Nadirashvili's construction technique, it does not seem to enable one to control the underlying complex structure on the examples, except of course in the simply-connected case when any (complete) bounded minimal such surface must be conformally equivalent to the disc $\D$. Indeed, the use of Runge's theorem at each step of the inductive construction does not allow to control the placement in $\R^3$ of some parts of the surface, and one has to cut away some small pieces of the surface in order to keep it suitably bounded (or contained in the given convex domain). This suffices to prescribe the topology of the examples but it does not enable to control their complex structure, and hence the so-called {\em conformal Calabi-Yau problem} asking which open Riemann surfaces (other than the disc) are the complex structure of a complete bounded minimal surface in $\R^3$ remained open.

Recall that a {\em bordered Riemann surface} is an open connected Riemann surface $M$ that is the interior, $M = \overline M \setminus bM$, of a compact one dimensional complex manifold $\overline M$ with smooth boundary $bM$ consisting of finitely many closed Jordan curves; such an $\overline M$ is a {\em compact bordered Riemann surface}. By \cite[Theorem 8.1]{Stout2007}, every bordered Riemann surface is conformally equivalent to a domain $M$ of the form $M=R\setminus \bigcup_{i=1}^kD_i$, where $R$ is a compact Riemann surface and $D_1,\ldots,D_k$ $(k\in\N)$ are mutually disjoint, smoothly bounded closed discs in $R$, by a map smoothly extending to the boundary.

The following summarizes the most relevant currently known positive results concerning the conformal Calabi-Yau problem.

%
%
\begin{theorem}\label{th:cCY}
\begin{enumerate}[\rm (i)]
\item Every bordered Riemann surface $M=R\setminus\bigcup_{i=1}^kD_i$ is the complex structure of a complete bounded minimal surface in $\R^3$.

\smallskip

\item Let $R$ be a compact Riemann surface and $M=R\setminus\bigcup_{i=1}^\infty D_i$ be a domain in $R$ whose complement is a countable union of mutually disjoint, smoothly bounded closed discs $D_i$ (i.e., diffeomorphic images of $\overline\D$). Then $M$ is the complex structure of a complete bounded minimal surface in $\R^3$. 

\smallskip

\item In every compact Riemann surface $R$ there is a Cantor set $C$ whose complement $M=R\setminus C$ is the complex structure of a complete bounded minimal surface in $\R^3$.

\smallskip

\item Let $R$ be a compact Riemann surface and $M=R\setminus (D\cup E)$ be a domain in $R$ where $E\subset R$ is compact and $D=\bigcup_{i=1}^\infty D_i$ is the union of a countable family of mutually disjoint, closed geometric discs.\footnote{A geometric disc in a compact Riemann surface $R$ is a topological disc whose lifts in the universal covering of $R$, which is $\D$, $\C$, or $\C\P^1$, is a round disc.} Fix a point $p_0\in M$ and denote $M_i=R\setminus \bigcup_{j=1}^iD_i$, $i\in \N$. If $\lim_{i\to\infty}\dist_{M_i}(p_0,E)=+\infty$ holds, then $M$ is the complex structure of a complete bounded minimal surface in $\R^3$. 
\end{enumerate}
\end{theorem}

Statements (i), (ii), and (iv) are due to Alarc\'on and Forstneri\v c \cite{AlarconForstneric2015MA,AlarconForstneric2021RMI}, while  (iii) is a result of Forstneri\v c \cite{Forstneric2022RMI}. We point out that the Cantor set in the construction in (iii) cannot be specified in advance. The distance $\dist_{M_i}(p_0,E)$ in (iv) is measured on paths in $M_i$ with respect to any given Riemannian metric on $R$. Note that the set $E$ in (iv) may have isolated points.

By the uniformization theorem of He and Schramm \cite{HeSchramm1993}, every open Riemann surface of finite genus and countably many ends is conformally equivalent to a {\em circle domain} in a compact Riemann surface $R$; i.e., a domain of the form $M=R\setminus\bigcup_iD_i$ whose complement is the union of countably many connected components $D_i$, each of which is a closed geometric disc or a point; see Stout \cite[Theorem 8.1]{Stout2007} for the case of finitely many ends. The ends $D_i$ of $M$ which are discs are called {\em disc ends}, while those which are points are called {\em point ends}. The following immediate corollary of statements (i) and (ii) in Theorem \ref{th:cCY} gives a complete solution to the conformal Calabi-Yau problem for surfaces of finite topology (statement (a)); the case of infinite topology is far from being fully understood, though. 
%
%
\begin{corollary}[\text{\cite[Corollary 7.4.7]{AlarconForstnericLopez2021Book}}]
\begin{enumerate}[\rm (a)]
\item An open Riemann surface of finite topology is the complex structure of a complete bounded minimal surface in $\R^3$ if and only if it has no point ends.

\smallskip

\item Every open Riemann surface of finite genus and countably many ends, all of which are disc ends, is the complex structure of a complete bounded minimal surface in $\R^3$.
\end{enumerate}
\end{corollary}

The main new ingredient in the proof of Theorem \ref{th:cCY} is the existence of approximate solutions to the {\em Riemann-Hilbert problem} for conformal minimal surfaces in $\R^3$ parameterized by an arbitrary compact bordered Riemann surface (see Theorem \ref{th:RH}). Indeed, replacing Runge's theorem in Nadirashvili's construction by suitable such solutions gives enough control on the placement of the whole surface in $\R^3$ to be able to avoid cutting at each step in the inductive process, thereby allowing to control the conformal structure of the examples. Another crucial tool is the method for exposing boundary points of a complex curve due to Forstneri\v c and Wold \cite{ForstnericWold2009}.

Assume that we are given a holomorphic map $f:\overline\D\to\C^n$ and a continuous map $g:b\D\times\overline\D\to\C^n$ such that $g(z,\cdot):\overline\D\to\C^n$ is holomorphic and $g(z,0)=f(z)$ for all $z\in b\D$. Given numbers $0<r<1$ and $\epsilon>0$, the classical approximate Riemann-Hilbert problem asks for finding a holomorphic map $F:\overline\D\to\C^n$ and a number $r'\in[r,1)$ satisfying the following properties:
\begin{enumerate}[A]
\item[$\bullet$] $|F(z)-f(z)|<\epsilon$ for all $z\in r'\overline\D$.

\smallskip

\item[$\bullet$] $\dist(F(z),g(z,b\D))<\epsilon$ for all $z\in b\D$.

\smallskip

\item[$\bullet$] $\dist(F(\rho z),g(z,\overline\D))<\epsilon$ for all $z\in b\D$ and $\rho\in [r',1]$.
\end{enumerate}
These conditions can be adapted to any compact bordered Riemann surface $\overline M$ in place of the disc $\overline\D$ as the domain of $f$ and $F$, and it is easy to see that the problem always has a solution; see, e.g., \cite{ForstnericGlobevnik2001MRL,DrinovecForstneric2012IUMJ} and also \cite{AlarconForstneric2015Abel}. The Riemann-Hilbert problem was introduced in the theory of minimal surfaces in $\R^3$ by Alarc\'on and Forstneri\v c in \cite{AlarconForstneric2015MA} (see also \cite{AlarconForstneric2015Abel}), leading to the proof of Theorem \ref{th:cCY} (i). It was later improved and extended for minimal surfaces in $\R^n$ with arbitrary $n\ge 3$ by Alarc\'on, Drinovec Drnov\v sek, Forstneri\v c, and L\'opez in \cite{AlarconDrinovecForstnericLopez2015PLMS,AlarconDrinovecForstnericLopez2019TAMS}. Let us record here the most precise version of this result in case $n=3$ that is currently available in the literature. Recall that the {\em flux} ${\rm Flux}(u)$ of a conformal minimal immersion $u:M\to\R^3$ is the cohomology class of its conjugate differential $d^cu=\imath(\bar\partial u-\partial u)$ in $H^1(M,\R^3)$; see, e.g., \cite[Definition 2.3.2]{AlarconForstnericLopez2021Book}. 
%
\begin{theorem}[\text{\cite[Theorem 6.4.1]{AlarconForstnericLopez2021Book}}]\label{th:RH}
Let $\overline M=M\cup bM$ be a compact bordered Riemann surface and $I_1,\ldots,I_k$ be mutually disjoint compact arcs in $bM$ that are not connected components of $bM$. Set $I=\bigcup_{i=1}^kI_k$ and choose an annular neighborhood $A\subset \overline M$ of $bM$ and a smooth retraction $\rho:A\to bM$. Also let $u:\overline M\to\R^3$ be a conformal minimal immersion of class $\Cscr^1(\overline M)$, $r:bM\to[0,1]$ be a continuous map with support in the relative interior of $I$, $\alpha:I\times\overline\D\to\R^3$ be a map of class $\Cscr^1$ such that $\alpha(z,\cdot):\overline\D\to\R^3$ is a conformal minimal immersion with $\alpha(z,0)=0$ for all $z\in I$, and $\Lambda\subset\mathring M$ be a finite set. Consider the map $\chi:bM\times\overline\D\to\R^3$ given by
\[
	\chi(z,\xi)=u(z)+\alpha(z,r(z)\xi),\quad z\in bM,\; \xi\in\overline\D,
\]
where $\alpha(z,r(z)\xi)=0$ for $z\in bM\setminus I$. Then, for any numbers $\epsilon>0$ and $d\in\N$ and any neighborhood $U\subset A$ of $I$ in $M$, there are a neighborhood $\Omega\subset U$ of ${\rm supp}(r)$ in $M$ and a conformal minimal immersion $\tilde u:\overline M\to\R^3$ such that the following hold:
\begin{enumerate}[A]
\item[$\bullet$] $|\tilde u(z)-u(z)|<\epsilon$ for all $z\in \overline M\setminus \Omega$.

\smallskip
\item[$\bullet$] $\dist(\tilde u(z),\chi(z,b\D))<\epsilon$ for all $z\in bM$.

\smallskip
\item[$\bullet$] $\dist(\tilde u(z),\chi(\rho(z),\overline\D))<\epsilon$ for all $z\in\Omega$.

\smallskip
\item[$\bullet$] $\tilde u-u$ vanishes to order $d$ at all points in $\Lambda$.

\smallskip
\item[$\bullet$] ${\rm Flux}(\tilde u)={\rm Flux}(u)$.
\end{enumerate}
\end{theorem}

Theorem \ref{th:RH}, combined with the aforementioned method for exposing boundary points of a complex curve by Forstneri\v c and Wold \cite{ForstnericWold2009} and the Mergelyan theorem for conformal minimal surfaces (see \cite[Theorem 3.6.1]{AlarconForstnericLopez2021Book}), has led to much more precise versions of Theorem \ref{th:cCY} which include approximation, jet-interpolation, hitting, prescription of flux, and control on the asymptotic behavior of the examples. We refer the interested readers to the mentioned sources as well as to \cite[Chapters 7 and 8]{AlarconForstnericLopez2021Book}, where analogous results for minimal surfaces in $\R^n$ for arbitrary $n\ge 3$ have been also established. In particular, concerning the possible {\em asymptotic behaviors} of the examples it is known that: 
\begin{enumerate}[(A)]
\item If $M=R\setminus\bigcup_iD_i$ is as in  (i) or (ii) in Theorem \ref{th:cCY}, then there is a continuous map $u:\overline M=M\cup bM\to\R^3$ such that the restriction $u|_M:M\to\R^3$ to $M$ is a complete conformal minimal immersion and the restriction $u|_{bM}:bM\to\R^3$ to the boundary $bM=\bigcup_ibD_i$ is an injective map, hence a topological embedding \cite{AlarconDrinovecForstnericLopez2015PLMS,AlarconForstneric2021RMI}. In particular, $u(bM)=\bigcup_iu(bD_i)$ is the disjoint union of countably many Jordan curves (finitely many if $M$ is a bordered Riemann surface). This generalises and uses ideas from the proof of a previous partial result by Mart\'in and Nadirashvili for the disc $M=\d$ \cite{MartinNadirashvili2007ARMA}. 

Likewise, if $M=R\setminus (D\cup E)$ is as in (iv) in Theorem \ref{th:cCY}, then there is a continuous map $u:\overline M\to\R^3$ such that $u|_M$ is a complete conformal minimal immersion and $u|_{bD}:bD=\bigcup_{i=1}^\infty bD_i\to\R^3$ is a topological embedding \cite{AlarconForstneric2021RMI}.

\smallskip

\item If $M=\overline M\setminus bM$ is a bordered Riemann surface, so as in (i) in Theorem \ref{th:cCY}, and $\Omega\subset \R^3$ is a minimally convex domain, then there is a complete {\em proper} conformal minimal immersion $u:M\to \Omega$; if the domain $\Omega$ is bounded, smoothly bounded, and strongly minimally convex, then $u$ can be chosen to be continuous on $\overline M=M\cup bM$ and hence to satisfy $u(bM)\subset b\Omega$  \cite{AlarconDrinovecForstnericLopez2019TAMS}. A domain in $\R^3$ is {\em (strongly) minimally convex} if it is (strongly) $2$-convex in the sense of Harvey and Lawson \cite[Definition 3.3]{HarveyLawson2013IUMJ}; see Definition \ref{def:MC}. 
We refer to \cite{HarveyLawson2013IUMJ}, \cite{AlarconDrinovecForstnericLopez2019TAMS}, and \cite[Chapter 8]{AlarconForstnericLopez2021Book} for a detailed discussion of such domains.

\smallskip

\item If $M$ is a bordered Riemann surface and $\Omega\subset \R^3$ is a domain, then there is a complete conformal minimal immersion $u:M\to \Omega$ with dense image: $\overline{u(M)}=\overline \Omega$; see Alarc\'on and Castro-Infantes \cite{AlarconCastroInfantes2018GT}. Note that no restriction at all is imposed on the domain $\Omega$.                        
\end{enumerate}

According to (A) and (B), there are complete bounded minimal surfaces in $\R^3$ with the complex structure of any given bordered Riemann surface and a nice asymptotic behavior. In view of (C), there are also such surfaces whose asymptotic behavior is rather wild. 
Less is known about the possible asymptotic behaviors of complete bounded minimal surfaces with the complex structure of a circle domain in a compact Riemann surface with countably infinitely many disc ends (as the domain $M$ in Theorem \ref{th:cCY} (ii)) or of those with the complex structure of the complement of a Cantor set in a compact Riemann surface (as $M=R\setminus C$ in Theorem \ref{th:cCY} (iii)). In particular, focusing on (B), the following questions remain open (even in the test case when the domain $\Omega$ is the ball $\b\subset\R^3$).
%
%
\begin{problem}\label{pr:}
Let $R$ be a compact Riemann surface and  $\Omega\subset \R^3$ be a minimally convex domain.
\begin{enumerate}[\rm (a)]
\item Let $M=R\setminus\bigcup_{i=1}^\infty D_i$ be a domain in $R$ as in Theorem \ref{th:cCY} (ii). Does there exist a complete proper conformal minimal immersion $M\to\Omega$ ?

\smallskip

\item Does there exist a Cantor set $C$ in $R$ whose complement admits a complete proper conformal minimal immersion $R\setminus C\to\Omega$ ?
\end{enumerate}
\end{problem}

In Section \ref{sec:th} we shall settle Problem \ref{pr:} (b) in the positive, thereby providing a new sampling of the type of constructions regarding the conformal Calabi-Yau problem that can be carried out by using the Riemann-Hilbert method. Our proof combines the arguments from \cite{AlarconDrinovecForstnericLopez2019TAMS} and \cite{Forstneric2022RMI}, and it is similar to that in the recent paper \cite{Castro-InfantesHidalgo2024} by Castro-Infantes and Hidalgo where the existence of properly immersed constant mean curvature one (CMC-1) surfaces in hyperbolic space and of almost proper CMC-1 faces in de Sitter space with the complex structure of the complement of a Cantor set in an arbitrary compact Riemann surface is established. Our result is also related to that by Alarc\'on and Forstneri\v c in \cite[Theorem 1.5]{AlarconForstneric2024Annali} to the effect that on every compact Riemann surface there is a Cantor set whose complement admits a proper holomorphic embedding into the affine plane $\C^2$; see also the paper by Di Salvo and Wold \cite{DiSalvoWold2022}.

We expect that Problem \ref{pr:} (a) admits an affirmative answer as well. It is likely not an easy question, though.


\section{Complete minimal surfaces with Cantor ends\\in minimally convex domains}\label{sec:th}

\noindent
An upper semicontinuous function $\phi:\Omega\to\R\cup\{-\infty\}$ on a domain $\Omega\subset\R^3$ is {\em minimal plurisubharmonic}, or {\em $2$-plurisubharmonic}, if the restriction $\phi|_{L\cap \Omega}$ of $\phi$ to any affine plane $L\subset\R^3$ is subharmonic. If a function $\phi:\Omega\to\R$ is of class $\Cscr^2$, then the {\em Hessian} of $\phi$ at a point $x=(x_1,x_2,x_3)\in\Omega$ is the quadratic form ${\rm Hess}_\phi(x)={\rm Hess}_\phi(x,\cdot)$ on the tangent space $T_x\R^3\cong\R^3$ given by
\[
	{\rm Hess}_\phi(x,\xi)=\sum_{j,k=1}^3\frac{\partial^2\phi}{\partial x_j\partial x_k}(x)\xi_j\xi_k,\quad \xi=(\xi_1,\xi_2,\xi_3)\in\R^3.
\]
A function $\phi$ of class $\Cscr^2(\Omega)$ is minimal plurisubharmonic if and only if $\lambda_1(x)+\lambda_2(x)\ge 0$ for all $x\in\Omega$, where $\lambda_1(x)\le \lambda_2(x)$ are the eigenvalues of ${\rm Hess}_\phi(x)$; and this happens if and only if $\phi|_S$ is subharmonic for every minimal surface $S\subset\Omega$ (see \cite[Proposition 2.3 and Theorem 2.13]{HarveyLawson2013IUMJ} or \cite[Proposition 2.2]{AlarconDrinovecForstnericLopez2019TAMS}). A function $\phi\in\Cscr^2(\Omega)$ is {\em minimal strongly plurisubharmonic}, or {\em strongly $2$-plurisubharmonic}, if $\lambda_1(x)+\lambda_2(x)> 0$ for all $x\in\Omega$, and this happens if and only if $\phi|_S$ is strongly subharmonic for every minimal surface $S\subset\Omega$ (see \cite[Proposition 2.4]{AlarconDrinovecForstnericLopez2019TAMS}).

%
%
\begin{definition}\label{def:MC}
A domain $\Omega\subset\R^3$ is {\em minimally convex}, or {\em $2$-convex}, if it admits a smooth minimal strongly plurisubharmonic exhaustion function.\footnote{A continuous function $\phi:\Omega\to\R$ that is bounded from below is an {\em exhaustion function} on $\Omega$ if it is a proper map; i.e, the closed sublevel set $\Omega_c=\{x\in\Omega\colon \phi(x)\le c\}$ is compact for all $c>0$.}

A bounded domain $\Omega\subset\R^3$ with boundary of class $\Cscr^2$ is {\em strongly minimally convex}, or {\em strongly $2$-convex}, if the principal curvatures $\kappa_1\le \kappa_2$ of $b\Omega$ from the interior side satisfy $\kappa_1(x)+\kappa_2(x)>0$ for all  $x\in b\Omega$.
\end{definition}

Every convex domain in $\R^3$ is minimally convex, but there are minimally convex domains without any convex boundary points. Moreover, a smoothly bounded domain $\Omega$ in $\R^3$ is minimally convex if and only if it is {\em mean-convex}, meaning that $\kappa_1(x)+\kappa_2(x)\ge 0$ for all  $x\in b\Omega$, where $\kappa_1\le \kappa_2$ are the principal curvatures  of $b\Omega$ from the interior side; but minimally convex domains need not be smoothly bounded. We refer to \cite{HarveyLawson2013IUMJ}, \cite{AlarconDrinovecForstnericLopez2019TAMS}, and \cite[Chapter 8]{AlarconForstnericLopez2021Book} for a detailed discussion about minimally convex and strongly minimally convex domains in $\R^3$ and, more generally, about $p$-convex and strongly $p$-convex domains in $\R^n$ for $n\ge 3$ and $p\in\{1,\ldots,n\}$.

Here is the main new result in this paper.
%
%
\begin{theorem}\label{th:}
Let $R$ be a compact Riemann surface, $\Omega\subset\R^3$ be a minimally convex domain, and $D\subset R$ be a smoothly bounded closed disc,  and  assume that we are given a conformal minimal immersion $u:R\setminus \mathring D\to\Omega$. 
Then, for any $\epsilon>0$ and any closed discrete subset $A\subset \Omega$ there are a Cantor set $C\subset \mathring D$ and a complete proper conformal minimal immersion $\tilde u:R\setminus C\to \Omega$ satistying the following conditions:
\begin{enumerate}[\rm (i)]
\item $|\tilde u(p)-u(p)|<\epsilon$ for all $p\in R\setminus \mathring D$.

\smallskip

\item $A\subset\tilde u(D\setminus C)$. 

\smallskip

\item The limit set of $\tilde u$ equals the boundary $b\Omega$ of $\Omega$.
\end{enumerate}
\end{theorem}
Recall that the limit set of an immersed surface $u:M\to\R^3$ is the set of points $x\in \R^3$ such that there is a divergent sequence $\{p_j\}_{j\in \N}$ in $M$ with $x=\lim_{j\to\infty} u(p_j)$. In particular, since the immersion $\tilde u:R\setminus C\to\Omega$ in the theorem is proper, its limit set equals $\overline{\tilde u(R\setminus C)}\setminus \tilde u(R\setminus C)\subset b\Omega$.

Theorem \ref{th:} is new even if we ask the given domain $\Omega$ to be convex; it is new even in the very special case when $\Omega$ is the unit ball $\b\subset\R^3$. The main tool in the proof of the theorem is the following simplified version of \cite[Lemma 8.4.6]{AlarconForstnericLopez2021Book}. Its proof relies on Theorem \ref{th:RH} together with the method for exposing boundary points of a complex curve in \cite{ForstnericWold2009} (see also \cite[Theorem 6.7.1]{AlarconForstnericLopez2021Book}) and the Mergelyan theorem for conformal minimal immersions in \cite[Theorem 3.6.1]{AlarconForstnericLopez2021Book}.
%
%
\begin{lemma}\label{lem:}
Let $\phi$ be a minimal strongly plurisubharmonic function on a domain $\Omega\subset\R^3$ and $a<b$ be numbers such that the set 
\[
	\Omega_{a,b}=\{x\in\Omega\colon a<\phi(x)<b\}
\]
is relatively compact in $\Omega$. Let $\overline M=M\cup bM$ be a compact bordered Riemann surface and $u:\overline M\to\Omega$ be a conformal minimal immersion of class $\Cscr^1(\overline M)$ such that $u(bM)\subset \Omega_{a,b}$. Also let $K\subset M$ be a compact set and fix a point $p_0\in K$. Then, for any finite set $\Lambda\subset M$ and any numbers $a<a'<b$, $\epsilon>0$, $\delta>0$, and $\mu>0$, there is a conformal minimal immersion $\tilde u:\overline M\to\Omega$ of class $\Cscr^1(\overline M)$ satisfying the following conditions.
\begin{enumerate}[A]
\item[$\bullet$] $a'<\phi(\tilde u(p))<b$ for all $p\in bM$.

\smallskip

\item[$\bullet$] $\phi(\tilde u(p))> \phi(u(p))-\delta$ for all $p\in\overline M$.

\smallskip

\item[$\bullet$] $|\tilde u(p)-u(p)|<\epsilon$ for all $p\in K$.

\smallskip

\item[$\bullet$] $\dist_{\tilde u}(p_0,bM)>\mu$.

\smallskip

\item[$\bullet$] $\tilde u(p)=u(p)$ for all $p\in\Lambda$.
\end{enumerate}
\end{lemma}

We also need to recall the following construction of a Cantor set in a  
domain $X_0\subset\C$. Here, we follow the exposition in \cite[Sect.\ 6]{AlarconForstneric2024Annali}. 
First, choose a smoothly bounded compact convex domain 
$\Delta_0\subset X_0$. Remove from $\Delta_0$ an open 
neighborhood $\Gamma_0$ of the vertical straight line segment dividing $\Delta_0$ into two convex sets of the same width in order to obtain two smoothly bounded compact convex subsets $\Delta_0^1$ 
and $\Delta_0^2$. Next, for $j=1,2$, remove from $\Delta_0^j$ 
an open neighborhood $\Gamma_0^j$ of the horizontal straight line segment 
dividing $\Delta_0^j$ in two convex subsets of the same height, in such a way that the two
components of $\Delta_0^j\setminus \Gamma_0^j$ are smoothly bounded 
compact convex domains. We thus obtain a compact set 
\begin{equation}\label{eq:X1}
	X_1=\Delta_0\setminus 
	(\Gamma_0\cup \Gamma_0^1\cup\Gamma_0^2)
	\subset X_0
\end{equation}
that is the union of four mutually disjoint, smoothly bounded compact convex  domains
$X_1^j$, $j=1,\ldots,4$.
In a second step, repeat the same process for each convex compact domain 
$X_1^j$, thereby getting four mutually disjoint smoothly 
bounded compact convex domains in its interior. Thus we have a compact set 
$X_2\subset\mathring X_1$ that is the union of sixteen smoothly bounded 
compact convex domains. Repeating this process in an inductive way, we get a decreasing sequence 
of smoothly bounded compact domains
\begin{equation}\label{eq:XXX}
	X_1\Supset X_2\Supset X_3 \Supset\cdots
\end{equation}
such that each $X_i$ is the union of $4^i$ mutually disjoint smoothly bounded 
compact convex domains. 
It turns out that the intersection
\begin{equation}\label{eq:Cantor}
	C = \bigcap_{i\ge 1} X_i
\end{equation}
is a Cantor set in $\C$ contained in $X_0$.
%
%
\begin{proof}[Proof of Theorem \ref{th:}]
Up to replacing $A$ by a larger closed discrete set in $\Omega$ if necessary, we assume without loss of generality that the closure $\overline A$ of $A$ in $\R^3$ contains $b\Omega$, so
\begin{equation}\label{eq:closureA}
	\overline A=A\cup b\Omega.
\end{equation}
Let $\phi:\Omega\to\R$ be a smooth positive minimal strongly plurisubharmonic Morse exhaustion function; such a function exists since $\Omega$ is minimally convex (see \cite[Remark 8.1.11]{AlarconForstnericLopez2021Book}). Let $0=r_0<r_1<r_2<\cdots$ be a divergent sequence of noncritical values of $\phi$ such that $\phi(u(p))<r_1$ for all $p\in R\setminus\mathring D$ and $A\cap\phi^{-1}(r_j)=\varnothing$ for all $j\in\N$. Denote by $L_j$ the connected component of $\Omega_{r_j}=\{x\in\Omega\colon \phi(x)\le r_j\}$ containing $\phi(u(R\setminus\mathring D))$, $j\in\N$. We thus have that
\begin{equation}\label{eq:L}
	\phi(u(R\setminus\mathring D)) \Subset L_1\Subset L_2\Subset\cdots\subset \bigcup_{j\in\N} L_j=\Omega
\end{equation}
is an exhaustion of $\Omega$ by connected smoothly bounded compact domains such that $A\cap bL_j=\varnothing$ for all $j\in\N$. Since $D$ is a smoothly bounded closed disc, it is a smoothly bounded compact convex domain in a holomorphic coordinate chart on $R$. Call $X_0=\mathring D$, $K_0=R\setminus X_0$,  $u_0=u$, $\epsilon_0=\epsilon$, and $L_0=\{x\in\Omega\colon \phi(x)\le r_0=0\}=\varnothing$, and let $\Delta_0\subset X_0$ be a smoothly bounded compact convex domain so large that $u_0$ extends to a conformal minimal immersion $u_0:R\setminus \mathring \Delta_0\to \Omega$ with the range in $\mathring L_1$. Extend $u_0$ to a generalized conformal minimal immersion $u_0:(R\setminus \mathring \Delta_0)\cup E_0\to \mathring L_1$ (see \cite[Definition 3.1.2]{AlarconForstnericLopez2021Book}) with $A\cap \mathring L_1\subset u_0(E_0)$, where $E_0$ is the vertical straight line segment dividing $\Delta_0$ in two convex subsets of the same width. By the Mergelyan theorem with interpolation for conformal minimal immersions (see \cite[Theorem 3.6.1]{AlarconForstnericLopez2021Book}), given a number $0<\epsilon_1<\epsilon_0/2$ there is a conformal minimal immersion $u_0':(R\setminus \mathring \Delta_0)\cup \overline\Gamma_0\to \mathring L_1$ such that $|u_0'(p)-u_0(p)|<\epsilon_1/3$ for all $p\in (R\setminus \mathring \Delta_0)\cup E_0$ and $A\cap \mathring L_1\subset u_0'(E_0)$, where $\Gamma_0$ is a neighborhood of $E_0$ as in the construction of a Cantor set above. Next, repeat the procedure simultaneously in the two connected components $\Delta_0^1$ and $\Delta_0^2$ of $\Delta_0\setminus \Gamma_0$ to find a compact set $X_1$ as in \eqref{eq:X1} and a conformal minimal immersion $u_0'':K_1=R\setminus \mathring X_1\to\mathring L_1$ such that $|u_0''(p)-u_0'(p)|<\epsilon_1/3$ for all $p\in (R\setminus \mathring \Delta_0)\cup \overline\Gamma_0$ and $A\cap \mathring L_1\subset u_0''(E_0)$. Finally, fix a point $p_0\in \mathring K_0$ and apply Lemma \ref{lem:} to obtain a conformal minimal immersion $u_1:K_1\to \mathring L_2$ such that $|u_1(p)-u_0''(p)|<\epsilon_1/3$ for all $p\in K_0$, $A\cap \mathring L_1\subset u_1(E_0)$, $u_1(bK_1)\subset \mathring L_2\setminus L_1$, and $\dist_{u_1}(p_0,bK_1)>1$. In particular, the following hold.
\begin{enumerate}[A]
\item[$\bullet$] $u_1:K_1=R\setminus \mathring X_1\to \Omega$ is a conformal minimal immersion, where $X_1\subset X_0$ is the union of $4$ mutually disjoint smoothly bounded compact convex domains.

\smallskip
\item[$\bullet$] $|u_1(p)-u_0(p)|<\epsilon_1$ for all $p\in K_0$.

\smallskip
\item[$\bullet$] $u_1(bK_1)\subset\mathring L_2\setminus L_1$.

\smallskip
\item[$\bullet$] $u_1(K_1\setminus \mathring K_0)\cap L_0=\varnothing$.

\smallskip
\item[$\bullet$] $\dist_{u_1}(p_0,bK_1)>1$. 

\smallskip
\item[$\bullet$] $A\cap \mathring L_1\subset u_1(\mathring K_1\setminus K_0)$.

\smallskip
\item[$\bullet$] $0<\epsilon_1<\epsilon_0/2$.
\end{enumerate}

Repeating this procedure based on Lemma \ref{lem:} and the Mergelyan theorem with interpolation for conformal minimal immersions \cite[Theorem 3.6.1]{AlarconForstnericLopez2021Book} in an inductive way, following the construction of a Cantor set $C$ in $X_0$ explained above (see \eqref{eq:XXX} and \eqref{eq:Cantor}), we may construct sequences $K_j$, $u_j$, $\epsilon_j$, $j\ge 2$, such that
\begin{enumerate}[\rm (a)]
\item $u_j:K_j=R\setminus \mathring X_j\to \Omega$ is a conformal minimal immersion, where $X_j\Subset X_{j-1}$ is the union of $4^j$ mutually disjoint smoothly bounded compact convex domains.

\smallskip
\item $|u_j(p)-u_{j-1}(p)|<\epsilon_j$ for all $p\in K_{j-1}$.

\smallskip
\item $u_j(bK_j)\subset\mathring L_{j+1}\setminus L_j$.

\smallskip
\item $u_i(K_i\setminus \mathring K_{i-1})\cap L_{i-1}=\varnothing$ for all $i=1,\ldots,j$. 

\smallskip
\item $\dist_{u_j}(p_0,bK_i)>i$ for all $i=1,\ldots,j$. 

\smallskip
\item $A\cap \mathring L_i\subset u_j(\mathring K_i\setminus K_{i-1})$ for all $i=1,\ldots,j$. 

\smallskip
\item $0<\epsilon_j<\epsilon_{j-1}/2$.
\end{enumerate}
It turns out that $K_0\Subset K_1\Subset K_2\Subset\cdots$ and
\[
	C=\bigcup_{j\ge 0} X_j=R\setminus \bigcup_{j\ge 0} K_j
\]
is a Cantor set in $R$. Choosing the number $\epsilon_j>0$ sufficiently small at each step, conditions (a), (b), and (g) ensure that there is a limit map 
\[
	\tilde u=\lim_{j\to\infty} u_j:R\setminus C= \bigcup_{j\ge 0} K_j\to\Omega
\]
which is a conformal minimal immersion with $|\tilde u(p)-u(p)|<\epsilon$ for all $p\in K_0=R\setminus \mathring D$; recall that $u=u_0$. Further, \eqref{eq:L} and (d) ensure that $\tilde u$ is a proper map, while (e) guarantees that $\tilde u$ is complete, and \eqref{eq:L} and (f) that $A\subset \tilde u(D\setminus C)$. Finally, this inclusion, \eqref{eq:closureA}, and the properness of $\tilde u$ ensure that the limit set of $\tilde u$ equals the boundary $b\Omega$ of $\Omega$.
\end{proof}


\subsection*{Acknowledgements}
Research partially supported by the State Research Agency (AEI) via the grants no.\ PID2020-117868GB-I00  and PID2023-150727NB-I00, and the ``Maria de Maeztu'' Unit of Excellence IMAG, reference CEX2020-001105-M, funded by MCIN/AEI/10.13039/501100011033/, Spain.



\medskip
\noindent Antonio Alarc\'{o}n
\newline
\noindent Departamento de Geometr\'{\i}a y Topolog\'{\i}a e Instituto de Matem\'aticas (IMAG), Universidad de Granada, Campus de Fuentenueva s/n, E--18071 Granada, Spain.
\newline
\noindent  e-mail: {\tt alarcon@ugr.es}

\end{document}